\documentclass[12pt]{amsart}
\usepackage{amssymb, latexsym}
\usepackage{xypic}
\usepackage{mathrsfs}

\def\supp{{\rm supp}}
\def\be{\begin{enumerate}}

\def\ee{\end{enumerate}}

\def\loc{{\rm\; loc\;}}

\newtheorem{thm}{Theorem}[section]

\newtheorem{cor}[thm]{Corollary}

\newtheorem{pro}[thm]{Proposition}

\newtheorem{lemma}[thm]{Lemma}

\newtheorem{de}[thm]{Definition}

\newtheorem{example}[thm]{Example}

\def\cal{\bf}
\begin{document}

\title{Joins and meets in  effect algebras}

\author{G. Bi\'nczak$^{1}$}

\address{$^1$: Faculty of Mathematics and Information Sciences
	Warsaw University of Technology
	00-662 Warsaw, Poland}
\author{J. Kaleta$^2$}

\address{$^2$ Department of Applied Mathematics
	 Warsaw University of Life Sciences 02-787 Warsaw, Poland}
\author{A. Zembrzuski$^3$}

\address{$^3$ Department of Informatics, Faculty of Applied Informatics and Mathematics,
	Warsaw University of Life Sciences 02-787 Warsaw, Poland}	

\email{$^1$  grzegorz.binczak@pw.edu.pl; $^2$ joanna\_kaleta@sggw.edu.pl;  $^3$  andrzej\_zembrzuski@sggw.edu.pl}



\begin{abstract}
 We know that each effect algebra  $E$ is isomorphic to $\pi(X)$ for some $E$-test spaces $(X,{\cal T})$.We describe when $\pi(x)\lor \pi(y)$ and $\pi(x)\land\pi(y)$
 exists for $x,y\in{\cal E}(X,{\cal T})$. Moreover we give the formula for $\pi(x)\lor\pi(x)$ and $\pi(x)\land\pi(y)$ using only $x,y$ and tests which are elements of ${\cal T}$.
We obtain an example of finite, not homogeneous effect algebra $E$ such that sharp elements of $E$ form a lattice, whereas $E$ is not a lattice.

\end{abstract}

\maketitle

\section{\bf Introduction}
Effect algebras have been introduced by Foulis and Bennet in 1994 (see \cite{FB94}) for the study of foundations of quantum mechanics (see \cite{DP00}). Independently, Chovanec and K\^opka introduced an essentially equivalent structure called $D$-{\em poset} 
(see \cite{KC94}). Another equivalent structure was introduced by Giuntini and Greuling in \cite{GG94}).

The most important example of an effect algebra is $(E(H),0,I,\oplus)$, where $H$ is a Hilbert space and $E(H)$ consists of all self-adjoint operators $A$ on $H$ such that $0\leq A\leq I$. For $A,B\in E(H)$, $A\oplus B$ is defined if and only if $A+B\leq I$ and then $A\oplus B=A+B$. Elements of $E(H)$ are called {\em effects} and they play an important role in the theory of quantum measurements (\cite{BLM91},\cite{BGL95}).

A quantum effect may be treated as two-valued (it means $0$ or $1$) quantum  measurement that may be unsharp (fuzzy).  If there exist some pairs of effects $a,b$ which posses an orthosum $a\oplus b$ then this orthosum correspond to a  parallel measurement of two effects. 

In this paper we describe  the join and meet in effect algebras. In order to do this we use algebraic $E$-test spaces in the following way:  we know that for every effect algebra $E$ there exists an algebraic $E$-test space $(X,{\cal T})$ such that $E$ is isomorphic to $\Pi(X)$ (please see \cite{G97}) so it is enough to describe joins and meets in $\Pi(X)$.

\begin{de}
In \cite{FB94} an {\em effect algebra} is defined to be an algebraic system $(E,0,1,\oplus)$  consisting of a set $E$, two special elements $0,1\in E$ called  the {\em zero} and the {\em unit}, and a partially defined binary operation $\oplus$ on $E$ that satisfies  the following conditions for all $p,q,r\in E$:
\be
\item{} [Commutative Law] If $p\oplus q$ is defined, then $q\oplus p$ is defined and $p\oplus q=q\oplus p$.
\item {}[Associative Law] If $q\oplus r$ is defined and $p\oplus(q\oplus r)$ is defined, then $p\oplus q$ is defined, $(p\oplus q)\oplus r$ is defined, and $p\oplus(q\oplus r)=(p\oplus q)\oplus r$.
\item {}[Orthosupplementation Law] For every $p\in E$ there exists a unique $q\in E$ such that $p\oplus q$ is defined and $p\oplus q=1$.
\item {}[Zero-unit Law] If $1\oplus p$ is defined, then $p=0$.
\ee
\end{de}

For simplicity, we often refer to $E$, rather than to $(E,0,1,\oplus)$, as being an effect algebra.
\begin{de}
If $p,q\in E$, we say that $p$ and $q$ are orthogonal and write $p\perp q$ iff $p\oplus q$ is defined in $E$. If $p,q\in E$ and $p\oplus q=1$, we call $q$ the {\em orthosupplement} of $p$ and write $p'=q$.
For effect algebras $E_1,E_2$ a mapping $\phi\colon E_1\to E_2$ is said to be an {\em isomorphism} if  $\phi$ is a bijection, 
$a\perp b\iff \phi(a)\perp\phi(b)$, $\phi(1)=1$ and $\phi(a\oplus b)=\phi(a)\oplus \phi(b)$.
\end{de}

It is shown in \cite{FB94} that the relation $\leq$ defined for $p,q\in E$ by $p\leq q$ iff $\exists r\in E$ with $p\oplus r=q$ is a partial order on $E$ and $0\leq p\leq 1$ holds for all $p\in E$. It is also shown that the mapping $p\mapsto p'$ is an order-reversing involution and that $q\perp p$ iff $q\leq p'$. Furthermore, $E$ satisfies the following {\em cancellation law}: If $p\oplus q\leq r\oplus q$, then $p\leq r$.

An element $a\in E$ is {\em sharp} if the greatest lower bound  of the set $\{a,a'\}$ equals $0$ (i.e. $a\land a'=0$). We denote the set of sharp elements of $E$ by $E_S$.

For $n\in{\mathbb N}$ and $x\in E$ let $nx=x\oplus x\oplus\ldots\oplus x$ ($n$-times). We say that \i$(x)=\max\{n\in {\mathbb N}\colon nx$ exists $\}$ is  {\em the isotropic index of} $x$.

An {\em atom} of an effect algebra $E$ is a minimal element of $E\setminus\{0\}$. An effect algebra $E$ is {\em atomic} if for every non-zero element $x\in E$ there exists an atom $a\in E$ such that $a\leq x$.

In \cite{G97} Gudder introduced algebraic $E$-test spaces:

\begin{de} Let $X$ be a nonempty set and ${\mathbb N}_0={\mathbb N}\cup\{0\}$. For $f,g\in{\mathbb N}_0^X$, we write $f\leq g$ if $f(x)\leq g(x)$ for all $x\in X$ and define
$f+g\in{\mathbb N}_0^X$ by $(f+g)(x)=f(x)+g(x)$. If $f\leq g$, we define $g-f\in{\mathbb N}_0^X$ by $(g-f)(x)=g(x)-f(x)$. Denote by $f_0\in{\mathbb N}_0^X$ the function $f_0(x)=0$ for all $x\in X$.

A pair $(X,\cal T)$ is an {\em  $E$-test spaces} if and only if ${\cal T}\subseteq N_0^X$ and the following conditions hold:
\be
\item For any $x\in X$ there exists a $t\in\cal T$ such that $t(x)\not=0$.
\item If $s,t\in\cal T$ with $s\leq t$, then $s=t$.
\ee
\end{de}

\begin{de}
Let ${\cal E}(X,{\cal T})=\{f\in{\mathbb N}_0^X\colon f\leq t$ for some $t\in\cal T\}$. Let $f,g,h\in{\cal E}(X,{\cal T})$ then 
\be
\item $f\perp g$ if $f+g\in{\cal E}(X,{\cal T})$,
\item $f\loc g$ if $f+g\in{\cal T}$,
\item $f\approx_hg$ if $f+h\in{\cal T}$ and $g+h\in{\cal T}$,
\item $f\approx g$ if there exists $h\in{\cal E}(X,{\cal T})$ such that $f\approx_hg$.
\ee
We say that $(X,{\cal T})$ is {\em algebraic} if for $f,g,h\in {\cal E}(X,{\cal T})$, $f\approx_hg$ and $h\perp f$ imply that $h\perp g$.
\end{de}
\begin{de}(see \cite{J03})
Let $E$ be a finite effect algebra. The {\em atomic effect test space of} $E$ is the pair $({\rm At}(E),{\cal T})$, where ${\rm At}(E)$ is the set of all atoms of $E$ and 
$${\cal T}=\{t\in{\mathbb N}_0^{E\setminus\{0\}}\colon \supp(t)\subseteq{\rm At}(E)\land \bigoplus_{a\in E}t(a)a=1\}$$
The elements of ${\cal T}$ are called {\em atomic tests} of $E$. We say that a mapping $f\in{\mathbb N}_0^{E\setminus\{0\}}$ is an {\em atomic event} of $E$ iff there is a $t\in{\cal T}$ such that $t\geq f$. For an atomic event $f$ of $E$, we write shortly $\bigoplus f$ instead of $\bigoplus_{a\in{\rm At}(E)}f(a)a$.
\end{de}

Let $(X,\cal T)$ be an algebraic $E$-test space. If $f\in{\cal E}(X,{\cal T})$, we define $\pi(f)=\{g\in{\cal E}(X,{\cal T})\colon g\approx f\}$ and
$$
\Pi=\Pi(X)=\{\pi(f)\colon f\in {\cal E}(X,{\cal T})\}
$$

We define $0,1\in\Pi$ by $0=\pi(f_0)=\{f_0\}$ and $1=\pi(t)$ for any $t\in\cal T$. We define $\pi(f)'=\pi(g)$ if $g\loc f$ (such $g$ exists since there exists $t\in\cal T$ such that $f\leq t$ and then $t-f\in{\cal E}(X,{\cal T})$ and $(t-f)\loc f$). We define $\pi(f)\oplus\pi(g)=\pi(f+g)$ when $f\perp g$.

\begin{thm} \cite[Theorem 3.2]{G97}\label{3.2}
If $(X,\cal T)$ is an algebraic $E$-test space, then $\Pi(X)$ can be organized into an effect algebra.
\end{thm}

\begin{thm}\cite[Theorem 3.3]{G97}\label{3.3}
If $P$ is an effect algebra, there exists an algebraic $E$-test space $(X,\cal T)$ such that $P$ is isomorphic to $\Pi(X)$.
\end{thm}

Thanks to the above theorems the study of  effect algebras can be reduced to the study of effect algebras of the form $\Pi(X)$, where $(X,{\cal T})$ is $E$-test space.

\section{Main Theorems}

Theorems \ref{jcon} and \ref{mcon} are the main theorems which one need to prove. In order to do this we need the following three technical lemmas:

\begin{lemma}\label{lem1}
Let $(X,\cal T)$ be an algebraic $E$-test space. Let $f,g\in  {\cal E}(X,{\cal T})$ then $\pi(f)=\pi(g)\iff\exists{t_1,t_2\in{\cal T}} f\leq t_1\land g\leq t_2\land t_1-f=t_2-g$.
\end{lemma}
\begin {proof}
Let   $(X,\cal T)$ be an algebraic $E$-test space and $f,g\in  {\cal E}(X,{\cal T})$. 

Assume that there exist $t_1,t_2\in{\cal T}$ such that $f\leq t_1$, $g\leq t_2$ and $t_1-f=t_2-g$. Let $h=t_1-f=t_2-g\in{\cal E}(X,{\cal T})$ then $f+h=t_1\in{\cal T}$ and $g+h=t_2\in{\cal T}$ so $f\approx_hg$ and $\pi(f)=\pi(g)$.

Assume that $\pi(f)=\pi(g)$ then there exists $h\in {\cal E}(X,{\cal T})$ such that $f\approx_hg$. It follows that $t_1=f+h\in{\cal T}$ and $t_2=g+h\in{\cal T}$. Then $f\leq t_1$, $g\leq t_2$ and $t_1-f=h=t_2-g$.
\end {proof}

\begin{lemma}\label{lem2}
Let $(X,\cal T)$ be an algebraic $E$-test space. Let $f,g\in  {\cal E}(X,{\cal T})$ then $\pi(f)\leq \pi(g)\iff\exists{h\in {\cal E}(X,{\cal T})}$  $\pi(h)=\pi(g)\land f\leq h$.
\end{lemma}
\begin {proof}
Let $(X,\cal T)$ be an algebraic $E$-test space. Let $f,g\in  {\cal E}(X,{\cal T})$ . 

Let $\pi(f)\leq\pi(g)$ than there exists $i\in {\cal E}(X,{\cal T})$ such that  $f\perp i$ and $\pi(f)\oplus\pi(i)=\pi(g)$ so $\pi(f+i)=\pi(g)$. Let $h=f+i\in{\cal E}(X,{\cal T})$. Then
$f\leq h$ and $\pi(h)=\pi(f+i)=\pi(g)$.

Assume that there exists $h\in {\cal E}(X,{\cal T})$ such that $\pi(h)=\pi(g)$ and $f\leq h$. Let $i=h-f\in {\cal E}(X,{\cal T})$. Then $h=f+i$ and $\pi(f)\oplus\pi(i)=\pi(f+i)=\pi(h)$ so $\pi(f)\leq\pi(h)=\pi(g)$.
\end {proof}

The above Lemma implies that if $f,g\in  {\cal E}(X,{\cal T})$ and $f\leq g$ (i.e. $f(x)\leq g(x)$ for all $x\in X$) then $\pi(f)\leq\pi(g)$.

\begin{lemma}\label{lem3}
Let $(X,\cal T)$ be an algebraic $E$-test space. Let $f,g\in  {\cal E}(X,{\cal T})$ then $\pi(f)\leq \pi(g)\iff\exists_{t_1,t_2\in{\cal T}}\;g\leq t_2\land f\leq (t_1+g)-t_2$
\end{lemma}
\begin {proof}
Let $f,g\in  {\cal E}(X,{\cal T})$ and $\pi(f)\leq\pi(g)$. By Lemma \ref{lem2} there exists $h\in  {\cal E}(X,{\cal T})$ such that $\pi(h)=\pi(g)$ and $f\leq h$. Then there exist $t_1,t_2\in{\cal T}$ such that $h\leq t_1$, $g\leq t_2$ and $t_1-h=t_2-g$ by Lemma \ref{lem1}. Hence $f\leq h=(t_1+g)-t_2$.

Now let $f,g\in  {\cal E}(X,{\cal T})$ and assume that there exist $t_1,t_2\in{\cal T}$ such that $g\leq t_2$ and $f\leq(t_1+g)-t_2$. Let $h=(t_1+g)-t_2$. Then
 $h=(t_1+g)-t_2\leq t_1+t_2-t_2=t_1$ so $h\in  {\cal E}(X,{\cal T})$. Moreover $\pi(h)=\pi(g)$ by Lemma \ref{lem1} since $h\leq t_1$, $g\leq t_2$ and
$t_1-h=t_1-((t_1+g)-t_2)=t_2-g$. Hence $\pi(f)\leq\pi(g)$ by Lemma \ref{lem2} since $\pi(h)=\pi(g)$ and $f\leq (t_1+g)-t_2=h$ which ends the proof.

\end {proof}

\begin{de}
Let $(X,\cal T)$ be an algebraic $E$-test space. Let $f,g\in  {\cal E}(X,{\cal T})$ then 
$$ub(f,g)=\{(f_1,f_2,f_3,f_4)\in{\cal T}^4\colon  
f\leq f_1, g\leq f_3, f-f_1+f_2\leq f_4, g-f_3+f_4\leq f_2\}$$
\end{de}
The set $ub(f,g)$  is non-empty since there exist $t_1,t_2\in{\cal T}$ such that $f\leq t_1$ and $g\leq t_2$ and then $(t_1,t_1,t_2,t_1)\in ub(f,g)$.

\begin{lemma}\label{lem4}
Let $ {\cal E}(X,{\cal T})$ be an algebraic $E$-test space and $f,g\in {\cal E}(X,{\cal T})$. For every $(f_1,f_2,f_3,f_4)\in ub(f,g)$ we have
\be
\item $i=\max(f-f_1+f_2,g-f_3+f_4,f_0)\in{\cal E}(X,{\cal T})$
\item $\pi(f)\leq\pi(i)$ and $\pi(g)\leq\pi(i)$
\ee

On the other hand if $h\in {\cal E}(X,{\cal T})$, $\pi(f)\leq\pi(h)$ and $\pi(g)\leq\pi(h)$ then
there exists $(f_1.f_2,f_3,f_4)\in ub(f,g)$ such that $i=\max(f-f_1+f_2,g-f_3+f_4,f_0)\leq h$.
 
\end{lemma}
\begin {proof}
Let $f,g\in {\cal E}(X,{\cal T})$ and  $(f_1,f_2,f_3,f_4)\in ub(f,g)$.
Let $i\in{\mathbb N}_0^X$ be a function such that $$i(x)=\max(f(x)-f_1(x)+f_2(x),g(x)-f_3(x)+f_4(x),0)$$ for every $x\in X$. Then 

$$f-f_1+f_2\leq i\eqno(1)$$
and
$$g-f_3+f_4\leq i\eqno(2)$$

We have $f(x)-f_1(x)+f_2(x)\leq f_1(x)-f_1(x)+f_2(x)=f_2(x)$ and $g(x)-f_3(x)+f_4(x)\leq f_2(x)$ since $(f_1,f_2,f_3,f_4)\in ub(f,g)$. Thus $0\leq i(x)\leq f_2(x)$ for every $x\in X$.

Hence $i\leq f_2\in{\cal T}$ so $i\in {\cal E}(X,{\cal T})$ and $\pi(f)\leq\pi(i)$ by Lemma \ref{lem3} since $i\leq f_2$ and $f\leq f_1+i-f_2$ by $(1)$. 
Since $(f_1,f_2,f_3,f_4)\in ub(f,g)$ therefore $f-f_1+f_2\leq f_4$ and $g-f_3+f_4\leq f_3-f_3+f_4=f_4$ thus $i\leq f_4$ and $g-f_3+f_4\leq i$ by $(2)$ so $g\leq (f_3+i)-f_4$ and $\pi(g)\leq\pi(i)$ by Lemma \ref{lem3}.
Therefore $\pi(i)$ is an upper bound of $\pi(f)$ and $\pi(g)$. 

Now assume that $h\in {\cal E}(X,{\cal T})$, $\pi(f)\leq\pi(h)$ and $\pi(g)\leq\pi(h)$. By Lemma \ref{lem3} there exist $f_1,f_2,f_3,f_4\in {\cal T}$ such that

$$
h\leq f_2,f\leq (f_1+h)-f_2,h\leq f_4,g\leq (f_3+h)-f_4.\eqno(3)
$$

Hence $ f\leq f_1-f_2+h\leq f_1-f_2+f_2=f_1$, $g\leq f_3-f_4+h\leq f_3-f_4+f_4=f_3$ and $f-f_1+f_2\leq h\leq f_4$, $g-f_3+f_4\leq h\leq f_2$ so $(f_1,f_2,f_3,f_4)\in ub(f,g)$.
Let $i\in{\mathbb N}_0^X$ be a function such that $$i(x)=\max(f(x)-f_1(x)+f_2(x),g(x)-f_3(x)+f_4(x),0)$$ for every $x\in X$. 
Moreover $f-f_1+f_2\leq h$ and $g-f_3+f_4\leq h$ by (3).  Hence by the definition of $i$ we have $i\leq h$.
\end {proof}

There is some connections between upper bounds of $\pi(f)$ and $\pi(g)$ and the set $ub(f,g)$. By lemma \ref{lem4} we can assign upper bound in the form $\pi(\max(f-f_1+f_2,g-f_3+f_4,f_0))$ of $\pi(f)$ and $\pi(g)$ to every $(f_1,f_2,f_3,f_4)\in ub(f,g)$. On the other hand one can assign $(f_1,f_2,f_3,f_4)\in ub(f,g)$ to every
 upper bound of $\pi(f)$ and $\pi(g)$
\begin{de}
Let $(X,\cal T)$ be an algebraic $E$-test space. Let $f,g\in  {\cal E}(X,{\cal T})$ then $f,g$ satisfy J-condition for $(f_1,f_2,f_3,f_4)\in{\cal T}^4$ if and only if 
\be
\item[\rm(J1)]
$$
f\leq f_1,\quad g\leq f_3,\quad f-f_1+f_2\leq f_4,\quad g-f_3+f_4\leq f_2
$$ and
\item[\rm(J2)] for every $g_1,g_2,g_3,g_4\in {\cal T}$ such that $(g_1,g_2,g_3,g_4)\in ub(f,g)$ i.e.
$$
f\leq g_1,\quad g\leq g_3,\quad f-g_1+g_2\leq g_4,\quad g-g_3+g_4\leq g_2
$$
we have $$\pi(\max(f-f_1+f_2,g-f_3+f_4,f_0))\leq\pi(\max(f-g_1+g_2,g-g_3+g_4,f_0)).$$ 
\ee

\end{de}

\begin{thm}\label{jcon}
Let $ {\cal E}(X,{\cal T})$ be an algebraic $E$-test space and $f,g\in {\cal E}(X,{\cal T})$ then $\pi(f)\lor\pi(g)$ exists in $\Pi(X)$ if and only if there exist $f_1,f_2,f_3,f_4\in{\cal T}$ such that  $f,g$ satisfy J-condition for $(f_1,f_2,f_3,f_4)\in{\cal T}^4$.
Moreover if  $f,g$ satisfy J-condition for $(f_1,f_2,f_3,f_4)\in ub(f,g)$. then
$$\pi(f)\lor\pi(g)=\pi(h),$$
where $h(x)=\max(f(x)-f_1(x)+f_2(x),g(x)-f_3(x)+f_4(x),0)$ for every $x\in X$.
\end{thm}
\begin {proof}
\be
\item[$\Rightarrow$]
Let $f,g\in {\cal E}(X,{\cal T})$ and $\pi(f)\lor\pi(g)=\pi(h)$ for some $h\in {\cal E}(X,{\cal T})$. Then $\pi(f)\leq\pi(h)$ and $\pi(g)\leq\pi(h)$. By Lemma \ref{lem4}
there exists $(f_1.f_2,f_3,f_4)\in ub(f,g)$ such that $$i_1=\max(f-f_1+f_2,g-f_3+f_4,f_0)\leq h.$$ We show that $f,g$  satisfy J-condition for $(f_1,f_2,f_3,f_4)\in{\cal T}^4$.

We know that $(f_1.f_2,f_3,f_4)\in ub(f,g)$ so 
$$
f\leq f_1,g\leq f_3, f-f_1+f_2\leq f_4,g-f_3+f_4\leq f_2
$$

and (J1) is fulfilled.

Now we show that $(J2)$ is also satisfied.

Assume that there exist $g_1,g_2,g_3,g_4\in {\cal T}$ such that
$$
f\leq g_1,g\leq g_3, f-g_1+g_2\leq g_4,g-g_3+g_4\leq g_2
$$
so $(g_1.g_2,g_3,g_4)\in ub(f,g)$.

Let $i_2\in{\mathbb N}_0^X$ be a function such that $$i_2(x)=\max(f(x)-g_1(x)+g_2(x),g(x)-g_3(x)+g_4(x),0)$$ for every $x\in X$. Then 

\be
\item $i_2=\max(f-g_1+g_2,g-g_3+g_4,f_0)\in{\cal E}(X,{\cal T})$,
\item $\pi(f)\leq\pi(i_2)$ and $\pi(g)\leq\pi(i_2)$
\ee

by Lemma \ref{lem4}.

Therefore $\pi(i_2)$ is an upper bound of $\pi(f)$ and $\pi(g)$. We know that $\pi(h)$ is a join of $\pi(f)$ and $\pi(g)$ so
$\pi(h)\leq \pi(i_2)$. By Lemma \ref{lem2} we know that $\pi(i_1)\leq\pi(h)$ since $i_1\leq h$. Thus $\pi(i_1)\leq\pi(i_2)$ i.e. 
$$\pi(\max(f-f_1+f_2,g-f_3+f_4,f_0))\leq\pi(\max(f-g_1+g_2,g-g_3+g_4,f_0))$$
so $f,g$ satisfy $J$-condition for $(f_1,f_2,f_3,f_4)\in{\cal T}^4$.
\item[$\Leftarrow$]

Assume that $f,g$ satisfy J-condition for $(f_1,f_2,f_3,f_4)\in {\cal T}^4$. Let 
$$i_1=\max(f-f_1+f_2,g-f_3+f_4,f_0).$$
 We show that $i_1\in{\cal E}(X,{\cal T})$ and  $\pi(f)\lor\pi(g)=\pi(i_1)$.

 By (J1) we have
$$
f\leq f_1,g\leq f_3, f-f_1+f_2\leq f_4,g-f_3+f_4\leq f_2
$$

so $(f_1,f_2,f_3,f_4)\in ub(f,g)$ and by Lemma \ref{lem4} we have
\be
\item $i_1\in{\cal E}(X,{\cal T})$,
\item $\pi(f)\leq\pi(i_1)$ and $\pi(g)\leq\pi(i_1)$.
\ee
Therefore $\pi(i_1)$ is an upper bound of $\pi(f) $ and $\pi(g)$.

Assume that $h\in{\cal E}(X,{\cal T})$,  $\pi(f)\leq\pi(h)$  and $\pi(g)\leq\pi(h)$. By lemma \ref{lem4} there exist $(g_1,g_2,g_3,g_4)\in ub(f,g)$ such that
$$i_2=\max(f-g_1+g_2,g-g_3+g_4,f_0)\leq h.$$
Hence $\pi(i_2)\leq\pi(h)$ by Lemma \ref{lem2}.

By (J2) we have

$$\pi(\max(f-f_1+f_2,g-f_3+f_4,f_0))\leq\pi(\max(f-g_1+g_2,g-g_3+g_4,f_0))$$ 
so $\pi(i_1)\leq\pi(i_2)$.

Hence $\pi(i_1)\leq\pi(i_2)\leq\pi(h)$ so $\pi(i_1)\leq\pi(h)$. Therefore $$\pi(i_1)=\pi(\max(f-f_1+f_2,g-f_3+f_4,f_0))=\pi(f)\lor\pi(g).$$
\ee
\end {proof}

Now we deal with the meets in effect algebras.

\begin{de}
Let $(X,\cal T)$ be an algebraic $E$-test space. Let $f,g\in  {\cal E}(X,{\cal T})$ then 
$$lb(f,g)=\{(f_1,f_2,f_3,f_4)\in{\cal T}^4\colon 
f\leq f_2, g\leq f_4,f-f_2+f_1\geq f_0, g-f_4+f_3\geq f_0\}$$
\end{de}
The set $lb(f,g)$ is non-empty since  there exist $t_1,t_2\in{\cal T}$ such that $f\leq t_1$ and $g\leq t_2$ and then $(t_1,t_1,t_2,t_2)\in lb(f,g)$. 

\begin{lemma}\label{lem5}
Let $ {\cal E}(X,{\cal T})$ be an algebraic $E$-test space and $f,g\in {\cal E}(X,{\cal T})$. For every $(f_1,f_2,f_3,f_4)\in lb(f,g)$ we have
\be
\item $i=\min(f-f_2+f_1,g-f_4+f_3)\in{\cal E}(X,{\cal T})$
\item $\pi(f)\geq\pi(i)$ and $\pi(g)\geq\pi(i)$
\ee

On the other hand if $h\in {\cal E}(X,{\cal T})$, $\pi(f)\geq\pi(h)$ and $\pi(g)\geq\pi(h)$ then
there exists $(f_1.f_2,f_3,f_4)\in lb(f,g)$ such that $i=\min(f-f_2+f_1,g-f_4+f_3,)\geq h$.
 
\end{lemma}
\begin {proof}
Let $f,g\in {\cal E}(X,{\cal T})$ and  $(f_1,f_2,f_3,f_4)\in lb(f,g)$.
Let $i\in{\mathbb N}_0^X$ be a function such that $$i(x)=\min(f(x)-f_2(x)+f_1(x),g(x)-f_4(x)+f_3(x))$$ for every $x\in X$. Then $i\leq f-f_2+f_1\leq f_1\in{\cal T}$ since $f\leq f_2$.
Moreover $f-f_2+f_1\geq f_0$, $g-f_4+f_3\geq f_0$ since $(f_1,f_2,f_3,f_4)\in lb(f,g)$ so $i\geq f_0$ and $i\in {\cal E}(X,{\cal T})$.

By the definition of $i$ we have $i\leq f- f_2+f_1$ and $i\leq g-f_4+f_3$.
Moreover $f\leq f_2$ and $g\leq f_4$ since $(f_1,f_2,f_3,f_4)\in lb(f,g)$. Thus $\pi(i)\leq\pi(f)$ and $\pi(i)\leq\pi(g)$  by Lemma \ref{lem3}.

Now assume that $h\in {\cal E}(X,{\cal T})$, $\pi(f)\geq\pi(h)$ and $\pi(g)\geq\pi(h)$. By Lemma \ref{lem3} there exist $f_1,f_2,f_3,f_4\in {\cal T}$ such that

$$
f\leq f_2,h\leq (f_1+f)-f_2,g\leq f_4,h\leq (f_3+g)-f_4
$$so

$$f\leq f_2, g\leq f_4,f-f_2+f_1\geq f_0, g-f_4+f_3\geq f_0$$and $(f_1,f_2,f_3,f_4)\in lb(f,g)$.

We know that $h\leq(f_1+f)-f_2$ and $h\leq(f_3+g)-f_4$ so $h\leq i=\min(f-f_2+f_1,g-f_4+f_3)$.
\end {proof}

\begin{de}
Let $(X,\cal T)$ be an algebraic $E$-test space. Let $f,g\in  {\cal E}(X,{\cal T})$ then $f,g$ satisfy M-condition for $(f_1,f_2,f_3,f_4)\in{\cal T}^4$ if and only if
\begin{enumerate}
\item[\rm(M1)]$$
f\leq f_2,\quad g\leq f_4,\quad f-f_2+f_1\geq f_0,\quad g-f_4+f_3\geq f_0
$$ and
\item[\rm(M2)] for every $g_1,g_2,g_3,g_4\in {\cal T}$ such that $(g_1,g_2,g_3,g_4)\in lb(f,g)$ i.e.
$$
f\leq g_2,\quad g\leq g_4,\quad f-g_2+g_1\geq f_0,\quad g-g_4+g_3\geq f_0
$$
we have $$\pi(\min(f-f_2+f_1,g-f_4+f_3))\geq\pi(\min(f-g_1+g_2,g-g_3+g_4)).$$ 
\end{enumerate}
\end{de}

\begin{thm}\label{mcon}
Let $ {\cal E}(X,{\cal T})$ be an algebraic $E$-test space and $f,g\in {\cal E}(X,{\cal T})$ then $\pi(f)\land\pi(g)$ exists in $\Pi(X)$ if and only if  there exist $f_1,f_2,f_3,f_4\in {\cal T}$ such that $f,g$ satisfy M-condition for $(f_1,f_2,f_3,f_4)\in{\cal T}$.
Moreover if  $f,g$ satisfy M-condition for $(f_1,f_2,f_3,f_4)\in lb(f,g)$ then
$$\pi(f)\land\pi(g)=\pi(h),$$
where $h(x)=\min(f(x)-f_2(x)+f_1(x),g(x)-f_4(x)+f_3(x))$ for every $x\in X$.
\end{thm}
\begin {proof}
\be
\item[$\Rightarrow$]
Let $f,g\in{\cal E}(X,{\cal T})$ and $\pi(f)\land\pi(g)=\pi(h)$ for some $h\in{\cal E}(X,{\cal T})$. Then $\pi(f)\geq\pi(h)$ and $\pi(g)\geq\pi(h)$ so by Lemma \ref{lem5}
there exists $(f_1.f_2,f_3,f_4)\in lb(f,g)$ such that $$i_1=\min(f-f_2+f_1,g-f_4+f_3,)\geq h.$$ We show that $f,g$ satisfy M-condition for $(f_1,f_2,f_3,f_4)\in{\cal T}^4$.

We know that $(f_1,f_2,f_3,f_4)\in lb(f,g)$ so

$$f\leq f_2, g\leq f_4,f-f_2+f_1\geq f_0, g-f_4+f_3\geq f_0$$

and (M1) is satisfied.

Now we show that (M2) is also satisfied.

Assume that there exist $g_1,g_2,g_3,g_4\in {\cal T}$ such that
$$
f\leq g_2,f_0\leq f+g_1-g_2,g\leq g_4,f_0\leq g+ g_3-g_4.
$$
so $(g_1,g_2,g_3,g_4)\in lb(f,g)$.

Let $i_2\in{\mathbb N}_0^X$ be a function such that

$$i_2(x)=\min(f(x)+g_1(x)-g_2(x),g(x)+g_3(x)-g_4(x))$$
for every $x\in X$. Then

\be
\item $i_2=\min(f-g_2+g_1,g-g_4+g_3)\in{\cal E}(X,{\cal T})$
\item $\pi(f)\geq\pi(i_2)$ and $\pi(g)\geq\pi(i_2)$
\ee

by Lemma \ref{lem5}.

Therefore $\pi(i_2)$ is a lower bound of  $\pi(f)$ and $\pi(g)$. We know that $\pi(h)$ is a meet of $\pi(f)$ and $\pi(g)$ so $\pi(i_2)\leq\pi(h)$. By Lemma~\ref{lem2} we know that
$\pi(i_1)\geq\pi(h)$ since $i_1\geq h$. Thus $\pi(i_2)\leq\pi(i_1)$ i.e.

 $$\pi(\min(f-f_2+f_1,g-f_4+f_3))\geq\pi(\min(f-g_1+g_2,g-g_3+g_4)).$$ 

so $f,g$ satisfies M-condition for $(f_1,f_2,f_3,f_4)\in{\cal T}^4$.

\item[$\Leftarrow$]

Assume that $f,g$ satisfy M-condition for $(f_1,f_2,f_3,f_4)\in {\cal T}^4$. Let  
$$i_1=\min(f+f_1-f_2,g+f_3-f_4).$$

 We show that $i_1\in{\cal E}(X,{\cal T})$ and  $\pi(f)\land\pi(g)=\pi(i_1)$. 

By (M1) we have

$$
f\leq f_2,\quad g\leq f_4,\quad f-f_2+f_1\geq f_0,\quad g-f_4+f_3\geq f_0
$$
so $(f_1,f_2,f_3,f_4)\in lb(f,g)$ and by Lemma~\ref{lem5} we have

\be
\item $i_1=\min(f-f_2+f_1,g-f_4+f_3)\in{\cal E}(X,{\cal T})$
\item $\pi(f)\geq\pi(i_1)$ and $\pi(g)\geq\pi(i_1)$
\ee

Therefore $\pi(i_1)$ is an lower bound of $\pi(f) $ and $\pi(g)$.

Assume that $h\in{\cal E}(X,{\cal T})$,  $\pi(h)\leq\pi(f)$  and $\pi(h)\leq\pi(g)$. By lemma \ref{lem5} there exists $(g_1,g_2,g_3,g_4)\in lb(f,g)$ such that
$$i_2=\min(f-g_1+g_2,g-g_3+g_4)\geq h.$$
Hence $\pi(i_2)\geq\pi(h)$ by Lemma~\ref{lem2}.

By (M2) we have

$$\pi(\min(f-f_2+f_1,g-f_4+f_3))\geq\pi(\min(f-g_1+g_2,g-g_3+g_4)).$$ 
so $\pi(i_1)\geq\pi(i_2)$. 

Hence $\pi(i_1)\geq\pi(i_2)\geq\pi(h)$ so $\pi(i_1)\geq\pi(h)$. Therefore

$$\pi(i_1)=\pi(\min(f+f_1-f_2,g+f_3-f_4)=\pi(f)\land\pi(g).$$
\ee
\end {proof}

In \cite{J01}, a new class of effect algebras, callled {\em homogeneous effect algebras} was introduced.

\begin{de}
An effect algebra is homogeneous if and only if, for all $u,u_1,u_2$ such that $u\leq u_1\oplus u_2\leq u'$ there are $v_1,v_2$ such that $u=v_1\oplus v_2$, $v_1\leq u_1$ and $v_2\leq u_2$.
\end{de}

The following Proposition  gives the simple way to check when a finite  effect algebra $E$ is homogeneous.

\begin{pro}\cite[Proposition 3.1]{J03}\label{homo} For every finite effect algebra $E$, the following are equivalent:
\be
\item[(a)] $E$ is homogeneous.
\item[(b)]  Let $u,f$ be a pair of atomic events such that $\bigoplus u\leq\bigoplus f\leq (\bigoplus u)'$. Then $u\leq f$.
\item[(c)] For every atom $a$ and for every atomic event $f$ such that $a\leq\bigoplus f\leq a'$, $a\in\supp(f)$
\item[(d)] Let $f,g$ be atomic tests, let $a\in\supp(f)\cap\supp(g)$. Then $f(a)=g(a)$.
\item[(e)] For every atom $a$ and every atomic event $f$ such that $a\in\supp(f)$, $f(a)=$\i$(a)$.
\ee

\end{pro}

Next Corollary gives the description of sharp elements in a homogeneous finite effect algebra $E$.

\begin{cor}\cite[Corollary 3.2]{J03} \label{sh} Let $E$ be a homogeneous effect algebra, let $w$ be an atomic event of $E$ such that $\bigoplus w\in E_S$. For every $a\in\supp(w)$, $w(a)=$\i$(a)$.
\end{cor}

It is easy to check that for finite homogeneous effect algebra $E$ and an atomic event $w$ of $E$ we have: $\bigoplus w\in E_S\iff $ for every $a\in\supp(w)$, $w(a)=$\i$(a)$.

In \cite{MS15} G.Jen\v ca posed the following problem: if $E$ is an orthocomplete homogeneous effect algebra $E$ such that $S(E)$ is a lattice, then $E$ is a lattice effect algebra.
In \cite{W18} Wei Ji gives an affirmative answer to this problem.

Using Theorem \ref{jcon}, Corollary \ref{sh} and Proposition \ref{homo} we created the algorithm checking if $E$ and $E_S$ are lattices. The following example of finite, not homogeneous effect algebra $E$ such that $E_S$ is a lattice  and $E$ is not a lattice was generated by computer making use of this algorithm. 

\begin{example}
Let $E$ be an effect algebra with the atomic test space given by the
following table
$$\begin{array}{|c|c|c|}\hline a&b&c\\\hline
2&2&0\\\hline1&0&2\\\hline
\end{array}$$
So in $E$ there are $3$ atoms $a,b,c$ and two atomic tests $t_1,t_2\colon E\setminus\{0\}\to{\mathbb N}_0$ such that $t_1(a)=t_1(b)=2$, $t_1(c)=0$ and $t_2(a)=1$, $t_2(b)=0$,
$t_2(c)=2$.

Then $E$ is not homogeneous by Theorem \ref{homo} since $t_2(a)=1\not=2=$\i$(a)$.

Moreover $E_S=\{1,0,2a,2b\}$. Of course, $E_S$ is a $4$-element lattice but $E$ is not a lattice since $a\lor c$ does not exist due to the fact that the set of upper bounds of $\{a,c\}$ consists of $1$, $a\oplus c$ and $2c$ so it does not have a smallest element.

 The effect algebra $E$ is represented by the following Hasse diagram:

$$
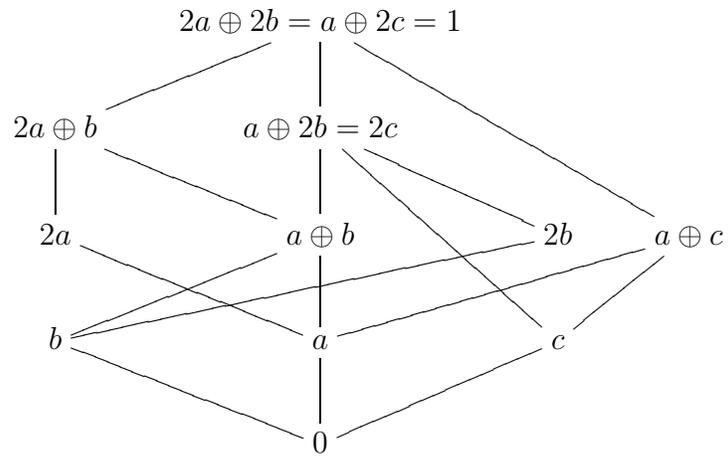
\xymatrix{&2a\oplus2b=a\oplus2c=1\ar@{-}[dl]\ar@{-}[d]\ar@{-}[ddrr]&&\\
2a\oplus b\ar@{-}[d]\ar@{-}[dr]&a\oplus2b=2c\ar@{-}[d]\ar@{-}[ddr]\ar@{-}[dr]&&\\
2a\ar@{-}[dr]&a\oplus b\ar@{-}[dl]\ar@{-}[d]&2b\ar@{-}[dll]&a\oplus c\ar@{-}[dll]\ar@{-}[dl]\\
b\ar@{-}[dr]&a\ar@{-}[d]&c\ar@{-}[dl]&\\
&0&&\\}
$$

\end{example}


\newpage
\begin{thebibliography}{99}
\bibitem{BLM91} Bush P., Lahti P.J., Mittelstadt P. The Quantum Theory of Measurement Lecture Notes in Phys. New Ser. m2, Springer-Verlag, Berlin, 1991

\bibitem{BGL95}  Bush P., Grabowski M., Lahti P.J Operational Quantum Physics, Springer-Verlag, Berlin, 1995

\bibitem{DP00} Dvure\v censkij A., Pulmannov\'a S. New Trends in Quantum Structures, Kluwer Academic Publ./Ister Science, Dordrecht-Boston-London/Bratislava, 2000.

\bibitem{FB94}   Foulis D. J., Bennett M. K.  Effect Algebras and Unsharp quantum Logics, Foundations of Physics. {\bf 24}, No. 10, 1994, 1331-1351.

\bibitem{GG94} Giuntini R., Grueuling H., Toward a formal language for unsharp properties, Found. Phys, {\bf 19}, 1994, 769-780.

\bibitem{G97} Gudder S. Effect test spaces and effect algebras, Foundations of Physics, {\bf 27} (2),1997, 287-304.

\bibitem{J01} G. Jen\v ca, Blocks of homogeneous effect algebras, Bull. Austal. Math. Soc., {\bf 64}, 2001, 81-98.

\bibitem{J03} Gejza Jen\v ca, Finite homogeneous and lattice ordered effect algebras, {\bf 272} (2-3), 2003, 197214

\bibitem{KC94} K\^opka F.,  Chovanec F., $D$-posets, Math. Slovaca, {\bf 44}, 1994, 21-34.

\bibitem{MS15} Mesiar R., Stup\v nanov\'a A., Open problems from the 12th International Conference on Fuzzy Set Theory and Its Applications, Fuzzy Sets and Systems, {\bf 261}, 112-123, 2015

\bibitem{W18} Wei Ji ,Meager projections in orthocomplete homogeneous effect algebras, Fuzzy Sets and Systems,{\bf 339}, 2018, 51-61

\end{thebibliography}
\end{document}